\documentclass[12pt,a4paper,final]{article}

\RequirePackage[backend=biber,style=authoryear,natbib,
  uniquename=false]{biblatex}

\usepackage{amsmath}    % AMS math style
\usepackage{amsfonts}   % Include AMS fonts
\usepackage{amssymb}    % Include AMS symbols
\usepackage{amsthm}     % For theorem-like environments
\usepackage{mathrsfs}   % Formal Script Math Symbol font
\usepackage{parskip}    % For whitespaced paragraphs
\usepackage{calc}       % For use of \widthof in def of \convas macro
\usepackage{enumitem}   % For nicer spacing in itemized environments
\usepackage{graphicx}   % For \includegraphics
\usepackage{lpic}       % For latex-annotated graphics
\usepackage[english]{babel}
\RequirePackage{hyperref}
\hypersetup{
    colorlinks=true, %set true if you want colored links
    linktoc=all,     %set to all for both sections and subsections linked
    linkcolor=blue,  %choose some color if you want links to stand out
    citecolor=blue,
    urlcolor = blue,
    }

\newlength{\vslength}
\newcommand{\ie}{{\it i.e.}}
\newcommand{\cf}{{\it c.f.}}
\newcommand{\eg}{{\it e.g.}}

\newcommand{\etc}{{\it etcetera}}

\newcommand{\scrB}{{\mathscr B}}

\newcommand{\scrG}{{\mathscr G}}

\newcommand{\scrP}{{\mathscr P}}

\newcommand{\scrX}{{\mathscr X}}

\renewcommand{\emptyset}{{\varnothing}}

\newcommand{\set}[1]{\left\{ #1 \right\}}

\newcommand{\ft}[2]{{\textstyle{\frac{#1}{#2}}}}

\newcommand{\conv}[1]%
  {{\mathrel{\,\xrightarrow{\widthof{\,#1\,}}\,}}}
\newcommand{\convas}[1]%
  {{\mathrel{\,\xrightarrow{\widthof{\,#1\text{-a.s.}\,}}\,}}}
\newcommand{\convprob}[1]%
  {{\mathrel{\,\xrightarrow{\widthof{\,#1\,}}\,}}}
\newcommand{\convweak}[1]%
  {{\mathrel{\,\xrightarrow{\widthof{\,#1\text{-w.}\,}}\,}}}

\newcommand{\twobytwo}[4]%
  {\left(\begin{array}{cc} #1 & #2 \\ #3 & #4 \end{array}\right)}
\newcommand{\twovec}[2]%
  {\left({\begin{array}{c} #1\\#2 \end{array}}\right)}

\newcommand{\ceiling}[1]{\left\lceil #1 \right\rceil}
\newcommand{\floor}[1]{\left\lfloor #1 \right\rfloor}

\renewcommand{\qedsymbol}{$\Box$}
\newcommand{\closebox}{{\hfill\qedsymbol}}

\newtheoremstyle{customtheorem}% name of the style to be used
  {0.5em}% measure of space to leave above the theorem. E.g.: 3pt
  {0.2em}% measure of space to leave below the theorem. E.g.: 3pt
  {\itshape}% name of font to use in the body of the theorem
  {}% measure of space to indent
  {\scshape}% name of head font
  {}% punctuation between head and body
  {1ex}% space after theorem head; " " = normal interword space
  {}% Manually specify head

\theoremstyle{customtheorem}
\newtheorem{theorem}{Theorem}[section]
\newtheorem{lemma}[theorem]{Lemma}
\newtheorem{proposition}[theorem]{Proposition}
\newtheorem{corollary}[theorem]{Corollary}
\newtheorem{definition}[theorem]{Definition}

\newtheoremstyle{customremark}% name of the style to be used
  {0.5em}% measure of space to leave above the theorem. E.g.: 3pt
  {0.2em}% measure of space to leave below the theorem. E.g.: 3pt
  {}% name of font to use in the body of the theorem
  {}% measure of space to indent
  {\scshape}% name of head font
  {}% punctuation between head and body
  {1ex}% space after theorem head; " " = normal interword space
  {}% Manually specify head

\theoremstyle{customremark}
\renewenvironment{proof}{\par\noindent{\scshape Proof}\;}{\closebox\par}
\newtheorem{remark}[theorem]{Remark}
\newtheorem{example}[theorem]{Example}

\newcommand{\comment}[1]{{}}
\newcommand{\extra}[1]{{}}

\AtEndDocument {
  \printbibliography
}

\begin{document}

\thispagestyle{empty}

\title{\vspace*{-16mm}
  Confidence sets in a sparse stochastic block model with two communities of
  unknown sizes
  }
\author{
  B.~J.~K.~{Kleijn}${}^{1}$ and J.~van~{Waaij}${}^{2}$\\[1mm]
  {\small\it ${}^{1}$ Korteweg-de~Vries Institute for Mathematics,
    University of Amsterdam}\\
  {\small\it ${}^{2}$ Department of Mathematical Sciences,
    University of Copenhagen}  
  }
\date{Aug 2021}
\maketitle

\begin{abstract}\noindent
In a sparse stochastic block model with two communities of unequal
sizes we derive two posterior concentration inequalities, that imply
(1) posterior (almost-)exact recovery of the community structure 
under sparsity bounds comparable to well-known sharp bounds in the
planted bi-section model; (2) a construction of confidence sets for
the community assignment from credible sets, with
finite graph sizes. The latter enables exact frequentist uncertain
quantification with Bayesian credible sets at non-asymptotic
graph sizes, where posteriors can be simulated well. There turns out
to be no proportionality between credible and confidence levels: for
given edge probabilities and a desired confidence level, there exists
a critical graph size where the required credible level drops sharply
from close to one to close to zero. At such graph sizes the
frequentist decides to include not most of the
posterior support for the construction of his confidence set, but
only a small subset of community assignments containing the highest
amounts of posterior probability (like the maximum-a-posteriori
estimator). It is argued that for the proposed construction of
confidence sets, a form of early stopping applies to MCMC sampling
of the posterior, which would enable the computation of confidence
sets at larger graph sizes.\\[.3em]
\end{abstract}

%%%%%%%%%%%%%%%%%%%%%%%%%%%%%%%%%%%%%%%%%%%%%%%%%%%%%%%%%%%%%%%%%%%%%%%%%%%%%%%

\section{Communities in sparse random graphs}
\label{sec:intro}

The stochastic block model \citep{Holland83} is an inhomogeneous version of the Erd\H os-R\'enyi random graph model \citep{Erdos59}: vertices belong to communities and edges occur independently with probabilities that depend on the communities of the vertices they connect. If we think of a resulting $n$-vertex random graph $X^n$ as data and the community assignments of the vertices as unobserved, a statistical challenge presents itself regarding estimation of the vertices' community assignments, a task referred to as \emph{community detection} \citep{Girvan02}. The stochastic block model and its generalizations have applications in physics, biology, sociology, image processing, genetics, medicine, logistics, \etc\ and are widely employed as canonical models for the study of clustering and community detection \citep{Fortunato10}.

In this paper, we consider sparse versions of the stochastic block model with two communities of unknown sizes (generalizing the so-called planted bi-section model \citep{Abbe18}). The main goal of this paper is to show that for given graph size and confidence level, credible sets for community assignments of high-enough credible level are (or can be enlarged to form) confidence sets. The derivation hinges on lower bounds for the expected posterior probability in (Hamming balls around) the true community assignment. These bounds are also sufficient to show that the posterior recovers community assignments consistently and, in that sense, are comparable to known sharp bounds in the stochastic block model with two equal communities \citep{Massoulie14,Abbe16,Mossel16}.

In subsection~\ref{sub:pbm} we discuss the literature on community detection for the planted bi-section model, focussing on necessary and sufficient conditions for exact and almost-exact recovery with varying degrees of edge sparsity. In subsection~\ref{sub:conclusions} we indicate pointwise which contributions this paper makes.

\subsection{The planted bi-section model}
\label{sub:pbm}

Most interest in the stochastic block model has come from network science and machine learning, in the form of a large number of algorithms that detect communities, with due attention for computational efficiency and scalability to large data sets. From the statistical perspective, algorithms for community detection are estimators for the unobserved community assignment. Estimation methods used for the community detection problem include spectral clustering (see \citep{Krzakala13} and many others), maximization of the likelihood and other modularities \citep{Girvan02,Bickel09,Choi12,Amini13}, semi-definite programming \citep{Hajek16,Guedon16}, and penalized ML detection of communities with minimax optimal mis-classification ratio \citep{Zhang16,Gao17}. More generally, we refer to \citep{Abbe18} and the informative introduction of \citep{Gao17} for extensive bibliographies and a more comprehensive discussion. Bayesian methods have been popular throughout, \eg\ the original work \citep{Snijders97}, the work of \citep{Decelle11a,Decelle11b} and, for example, \citep{Suwan16} based on an empirical prior choice. MCMC simulation of posteriors is discussed, for example, in \citep{mcdaid13,Geng19,Jiang21}.

Over the last decade there has also been a great interest in asymptotic lower bounds for edge sparsity that leave consistent community detection (only just) possible as the graph size $n$ grows. Particularly, which conditions on edge probabilities enable estimation of the true community assignments correctly with high probability? (\emph{exact recovery}, see definition~\ref{def:exact}); or correctly for all but a (possibly vanishing) fraction of the vertices with high probability (\emph{almost-exact recovery} with a certain error-rate, see definition~\ref{def:detect}). In \citep{Dyer89,Decelle11a,Decelle11b,Abbe16,Massoulie14,Mossel16} and many other publications, asymptotic limitations on the estimation problem are studied in the context of the so-called \emph{planted bi-section model}, which is a stochastic block model with two equally-sized communities of $n$ vertices each and edge probabilities $p_n$ (within communities) and $q_n$ (between communities).

The planted bi-section model with edge probabilities $p_n=a_n\log(n)/n$, $q_n=b_n\log(n)/n$ and $a_n,b_n=O(1)$ (the so-called Chernoff-Hellinger sparsity phase, in which expected degrees grow logarithmically with $n$) was considered in \citep{Massoulie14,Mossel15,Mossel16,Abbe16}: assuming that $a_n, b_n$ stay bounded away from zero and infinity, the communities in the planted bi-section graph with $2n$ vertices can be \emph{recovered exactly}, if and only if,
\begin{equation} \label{eq:mnscritical}
  \Bigl(\bigl(\sqrt{a_n}-\sqrt{b_n}\bigr)^2-2\Bigr)\log(n)
    + \log(\log(n))\to \infty,
\end{equation}
(see \citep{Mossel16}). With $p_n=c_n/n$, $q_n=d_n/n$ and $c_n,d_n=o(\log(n))$ (the so-called Kesten-Stigum sparsity phase of the problem, typically with $c_n,d_n=O(1)$ which keeps the expected degree of vertices bounded in the limit), \citep{Decelle11a,Decelle11b} conjectured that almost-exact recovery is possible in the planted bi-section model, if and only if,
\begin{equation} \label{eq:decellescondition}
  \frac{(c_n-d_n)^2}{2(c_n+d_n)} > 1.
\end{equation}
Additionally, \citep{Mossel16} prove that almost-exact recovery with a \emph{vanishing} fraction of possible mis-assignments (termed \emph{weak consistency} \citep{Mossel16}) is possible (by any estimator or algorithm), if and only if,
\begin{equation} \label{eq:MNSdetect}
  \frac{(c_n-d_n)^2}{2(c_n+d_n)}\to\infty.
\end{equation}
Conditions (\ref{eq:mnscritical})--(\ref{eq:MNSdetect}) are not only there to lower-bound the sparsity of edges in an absolute sense, but also guarantee sufficient separation \citep{Banerjee18} from the Erd\H os-R\'enyi graph ($p_n=q_n$) in which communities are not statistically identifiable.

\subsection{Posterior convergence and confidence sets for
community assignments}
\label{sub:conclusions}

In this paper we continue the study of sparse stochastic block models with two communities, but we generalize the assumption that both communities are of equal sizes; any two sizes that add up to $n$ vertices are permitted. In section~\ref{sec:commrecovery} we derive bounds for exact and almost-exact recovery with posteriors. In the Chernoff-Hellinger phase,
\begin{itemize}
\item[(i.)] the condition on edge sparsity for exact recovery is an analogue of condition (\ref{eq:mnscritical}) that takes into account the fact that community sizes are unknown (see corollary~\ref{cor:mnscritical}).
\end{itemize}
In the Kersten-Stigum phase we derive a sharp lower bound for the posterior mass in Hamming balls centred on the true community assignments of radii $k_n\geq a_nn$, leading to a condition that relates edge sparsity and error rates,
\begin{equation}
  \label{eq:KvWcritical}
  a_nn\Bigl(\log(a_n)+\frac14(\sqrt{c_n}
    -\sqrt{d_n})^2-1\Bigr)\to\infty.
\end{equation}
% Based on condition (\ref{eq:KvWcritical}), it is shown that the
% fact that community sizes are unknown raises the lower-bounding
% constant in (\ref{eq:decellescondition}) by a factor four, and that
% the lower bound is raised further by the (fixed) error rate.
% Regarding situations where the fraction of mis-assigned vertices
% is vanishing, condition (\ref{eq:MNSdetect}) remains unchanged
% even when the community sizes are unknown. To summarize,
With~(\ref{eq:KvWcritical}) it is shown:
\begin{itemize}
\item[(ii.)] by how much the lower bound in condition~(\ref{eq:decellescondition}) has to be raised to characterize almost-exact recovery with a non-vanishing fraction of mis-assigned vertices and unknown community sizes (see corollary~\ref{cor:decellescondition});
\item[(iii.)] that the limit (\ref{eq:MNSdetect}) continues characterize almost-exact recovery with a vanishing fraction of mis-assigned vertices when community sizes are unknown (see corollary~\ref{cor:MNSdetect});
\item[(iv.)] that in any situation in which posterior almost-exact recovery with error rates as small as $O(\log(n))$ is possible, the posterior recovers the community assignment exactly (see example~\ref{ex:lognerrorrate}).
\end{itemize}

Calculation, approximation or simulation of a posterior distribution is considered computationally costly; if the statistical goal is only the estimation of the community assignments, more efficient algorithms are known. However, the lack of sampling distributions for said efficient algorithms makes answering more complex statistical questions (like uncertainty quantification and testing of hypotheses) prohibitively hard. The second contribution in this paper is a detailed demonstration that frequentist uncertainty quantification can be based on the posterior distribution \emph{at finite values of the graph size $n$}. More particularly, in section~\ref{sec:pbmuncertainty} it is shown that:
\begin{itemize}
\item[(vi.)] posterior exact recovery (as in theorem~\ref{thm:exactrecovery}) permits the interpretation of credible sets as confidence sets, with a lower bound for the credible level in terms of the desired confidence level (see proposition~\ref{prop:exactcredconf});
\item[(vii.)] posterior almost-exact recovery with an error rate $k_n$ (as in theorem~\ref{thm:almostexactrecovery}) enables interpretation of $k_n$-enlarged credible sets as confidence sets, again with a lower bound for the credible level in terms of the desired confidence level (see proposition~\ref{prop:almostexactcredconf}).
\end{itemize}
As it turns out, there is no proportionality between a desired confidence level and the required credible level for a credible set (or its enlargement) to be a confidence set of said desired level. The relationship is more complex and revolves around the bounds derived in section~\ref{sec:commrecovery}:
\begin{itemize}
\item[(viii.)] for given edge probabilities $p,q$ and desired confidence level $\alpha$, a \emph{critical graph size} $n(p,q;\alpha)$ exists that distinguishes between cases in which credible sets of relatively low credible level can serve as confidence sets, and when it is required to use credible sets of relatively high credible level (see figure~\ref{fig:crediblelevel}).
\end{itemize}
When the graph size lies above its critical value, the frequentist decides to include not most of the posterior support for the construction of his confidence set, but only a small subset of community assignments containing the highest amounts of posterior probability (like the maximum-a-posteriori estimator). In discussion section~\ref{sec:discussion}, the latter point is used to argue that a form of early stopping in the MCMC sampling of the posterior may give rise to confidence sets at large graph sizes.

\subsection*{Acknowledgements} The authors thank E.~Mossel and J.~Neeman for a helpful discussion on necessary conditions for exact recovery. BK thanks P. Bickel for his encouragement to pursue the confidence-sets-from-credible-sets question.

%%%%%%%%%%%%%%%%%%%%%%%%%%%%%%%%%%%%%%%%%%%%%%%%%%%%%%%%%%%%%%%%%%%%%%%%%%%%%%%
\section{The sparse two-community stochastic block model}
\label{sec:pbm}

In the general stochastic block model, $n\geq1$ vertices are assigned to $K\geq2$ communities with an unobserved \emph{community assignment vector} $\theta_n=(\theta_{n,1},\ldots,\theta_{n,n})$, $\theta_{n,i}\in\{0,\ldots,K-1\}$ for $1\leq i\leq n$. The observation is a set $X^n=\{X_{ij}:1\leq i< j\leq n\}$ of undirected edges (with no self-loops), each of which occur independently with probabilities that depend on the communities of the vertices they connect. Our statistical goal is inference on $\theta_n$ using $X^n$, in block model with edges that become increasingly sparse with growing $n$, \eg\ with asymptotic degrees that stay bounded or grow only as $\log(n)$.

In the \emph{planted bi-section model} of \citep{Dyer89,Decelle11a,Decelle11b,Abbe16,Massoulie14,Mossel16}, $K=2$ and the two communities have equal sizes. We generalize to community assignments where one community (the smallest) has $0\leq m\leq\floor{n/2}$ vertices (denoted $m_\theta$ when the underlying community assignment $\theta$ is of importance) and the other (the largest) has $n-m$. Community assignments $\theta_{n,i}$ are either $0$ or $1$ (for the largest and smallest communities respectively). The parameter space $\theta_n$ can be written as a union,
\[
  \Theta_n=\bigcup_{m=0}^{\floor{n/2}}\Theta_{n,m},
\]
where $\Theta_{n,m}$ denotes the set of those $\theta_n\in\{0,1\}^{n}$ with $\Sigma_i\theta_{n,i}=m$, which has $\binom{n}{m}$ elements. (For even $n$ there is a note of identifiability: because, as we shall see later, $\theta_n=(\theta_{n,1},\ldots,\theta_{n,n})$ and $(1-\theta_{n,1},\ldots,1-\theta_{n,n})$ (notation $1-\theta_n$) induce the same law for $X^n$, identifiability is guaranteed if we define $\Theta_{n,n/2}=\{\theta_n:\Sigma_i\theta_{n,i}=n/2$ and $\theta_1=0$\}, and $\Theta_{n,n/2}$ has $\frac12\binom{n}{n/2}$ elements.) The full parameter set $\Theta_n$ has $2^{n-1}$ elements. It is noted that \(m=0\) is allowed (an Erd\H os-R\'enyi graph displaying no community structure).

The random graph $X^n$ takes its values in a space $\scrX_n$ with law $P_{\theta_n}$ under $\theta_n\in\Theta_n$. The ($n$-dependent) probability of an edge between vertices \emph{within a community} is denoted $p_n\in[0,1]$; the ($n$-dependent) probability of an edge \emph{between communities} is denoted $q_n\in[0,1]$,
\begin{equation}
  \label{eq:pbm}
  Q_{ij}(\theta_n):=P_{\theta_n}(X_{ij}=1)=\begin{cases}
    \,\,p_n,&\quad\text{if $\theta_{n,i}=\theta_{n,j}$,}\\
    \,\,q_n,&\quad\text{if $\theta_{n,i}\neq\theta_{n,j}$.}
  \end{cases}
\end{equation}
Edge sparsity distinguishes the \emph{Chernoff-Hellinger phase} of the model (where we take $a_n,b_n=O(1)$ and $p_n=a_nn^{-1}\log n$, $q_n=b_nn^{-1}\log n$) and the sparser \emph{Kesten-Stigum phase} (where we take $c_n,d_n=O(1)$ (or at most $o(\log(n))$) and $p_n=c_nn^{-1}$, $q_n=d_nn^{-1}$). Given $\theta\in\Theta_n$, the probability density for $P_{\theta_n}$ at $x^n\in\scrX_n$ is given by $p_{\theta_n}(x^n)=\prod_{i<j} Q_{ij}(\theta_n)^{x_{ij}}(1-Q_{ij}(\theta_n))^{1-x_{ij}}$. Asymptotically the first statistical question in this model concerns estimation of the community assignments $\theta_n$ in consistent ways, that is, (close to) correctly with probability growing to one as $n\to\infty$. In the Chernoff-Hellinger phase a suitable formulation of consistency is the following.
\begin{definition}
\label{def:exact}
Given community assignments $\theta_n$ for all $n\geq1$, an estimator sequence $\hat{\theta}_n:\scrX_n\to\Theta_n$ is said to \emph{recover $\theta_n$ exactly} if $\hat{\theta}_n$ is correct with high probability, \ie,
\[
  P_{\theta_n}\bigl(\,\hat{\theta}_n(X^n)=\theta_n\,\bigr)\to1,
\]
as $n\to\infty$. 
\end{definition}
In the Kesten-Stigum phase the appropriate form of consistency is more diffuse: rather than looking for exact matches, we allow for controlled differences between the estimated and true community assignments. For two sequences $\theta_n,\eta_n\in\Theta_n$, the so-called \emph{Hamming distance} $h$ denotes the number of differing bits, that is: $h(\theta_n,\eta_n)= \Sigma_i|\theta_{n,i}-\eta_{n,i}|$. Since $\theta_n$ and $1-\theta_n$ induce the same law for $X^n$, $\theta_n$ is considered close to $\eta_n$ when either $h(\theta_n,\eta_n)$ or $n-h(\theta_n,\eta_n)$ is small. This motivates the following definition,
\begin{equation}
  \label{eq:kmetric}
  k(\theta_n,\eta_n) = h(\theta_n,\eta_n)\wedge\bigl(n-h(\theta_n,\eta_n)\bigr),
\end{equation}
which defines a metric on $\Theta_n$.
\begin{definition}
\label{def:detect}
Let $\theta_n\in\Theta_n$ and some sequence of positive integers $(k_n)$ of order $k_n=O(n)$ be given. An estimator sequence $\hat{\theta}_n:\scrX_n\to\Theta_n$ is said to \emph{recover $\theta_{0,n}$ almost-exactly with error rate $k_n$}, if,
\[
  P_{\theta_{n}}\bigl(\,
    k\bigl(\hat{\theta}_{n}(X^n),\theta_{n}\bigr) \leq k_n
    \,\bigr)\to 1.
\]
\end{definition}
Note that $0\leq k(\theta_n,\eta_n)\leq n/2$ for any $\theta_,\eta_n\in\Theta_n$, so the error-rate must satisfy $0\leq k_n\leq n/2$.

%%%%%%%%%%%%%%%%%%%%%%%%%%%%%%%%%%%%%%%%%%%%%%%%%%%%%%%%%%%%%%%%%%%%%%%%%%%%%%%

\section{%Priors for community assignments and
  Posterior concentration}
\label{sec:postconcentration} % was {sec:somepriors}

In what follows we specialize to the Bayesian approach: we choose prior distributions $\Pi_n$ on $\Theta_n$ for all $n\geq1$, denoting probability mass functions by $\pi_n:\Theta_n\to[0,1]$. Throughout we assume that for all $\theta_n\in\Theta_n$, $\pi_n(\theta_n)>0$. In later sections we specialize to uniform priors: for every $n\geq1$ and every $\theta_n\in\Theta_n$, $\pi_n(\theta_n)=|\Theta_n|^{-1}=2^{-(n-1)}$.

The posterior for a set $A\subset\Theta_n$ is calculated,
\[
  \Pi(A|X^n)={\displaystyle \sum_{\theta_n\in A}
    p_{\theta_n}(X^n)\, \pi_n(\theta_n)}
    \biggm/
  {\displaystyle \sum_{\theta'_n\in\Theta_n}
    p_{\theta'_n}(X^n)\, \pi_n(\theta'_n)}.
\]
The central upper bound on posterior mass for sets of the type relevant in definitions~\ref{def:exact} and~\ref{def:detect} is given in proposition~\ref{prop:postconvset}, which makes use of the following definitions: fix $n\geq1$ and for $\theta_n,\eta_n\in\Theta_n$, define,
\begin{equation} \label{eq:thesetsD}
  \begin{split}
    D_{1}(\theta_n,\eta_n)&=\{(i,j)\in\{1,\ldots,n\}^2:\,i<j,\,
      \theta_{n,i}=\theta_{n,j},\,
      \eta_{n,i}\neq\eta_{n,j}\},\\
    D_{2}(\theta_n,\eta_n)&=\{(i,j)\in\{1,\ldots,n\}^2:\,i<j,\,
      \theta_{n,i}\neq\theta_{n,j},\,
      \eta_{n,i}=\eta_{n,j}\}.
  \end{split}
\end{equation}
The number $D_1$ is the number of edges (from the complete graph with $n$ vertices) whose probabilities change from $p_n$ to $q_n$ upon replacement of $\theta_n$ with $\eta_n$ (and $D_{2}$ how many edges change probabilities from $q_n$ to $p_n$). Note that the total number of edges that change probabilities is given by $|D_1\cup D_2|=|D_1|+|D_2|$. Furthermore, let,
\begin{equation}
  \label{eq:AffBernoulli}
  \rho(p,q)=p^{1/2}q^{1/2}+(1-p)^{1/2}(1-q)^{1/2},
\end{equation}
denote the Hellinger-affinity between two Bernoulli-distributions with parameters $p,q\in(0,1)$.
\begin{proposition}
\label{prop:postconvset}
Fix $n\geq2$ and a prior probability mass function $\pi_n$ on $\Theta_n$ of full support. Suppose that for some $\theta_n\in\Theta_n$, we observe a graph $X^n\sim P_{\theta_n}$. Let $S_n\subset\Theta_n\setminus\{\theta_n\}$ be non-empty. Then,
\begin{itemize}
\item[(1.)] the number of edge probability changes $d_n$ is lower bounded,
  \begin{equation}
    \label{eq:B-lowerbound}
    d_n=\min_{\eta_n\in S_n}|D_{1}(\theta_n,\eta_n)\cup D_{2}(\theta_n,\eta_n)|
      \geq \min_{\eta_n\in S_n}|m_{\theta_n}-m_{\eta_n}|(n-|m_{\theta_n}-m_{\eta_n}|),
  \end{equation}
\item[(2.)] the posterior mass of $S_n$ satisfies the upper bound, 
  \begin{equation}
    \label{eq:ordervstestpwr}
    P_{\theta_n}
    \Pi_n\bigl(S_n\bigm|X^n\bigr)
    \leq \rho(p_n,q_n)^{d_n}\sum_{\eta_n\in S_n}
    \sqrt{\frac{\pi_n(\eta_n)}{\pi_n(\theta_n)}}.
  \end{equation}
\end{itemize}
\end{proposition}
\begin{proof}
Because $\eta_n$ and $\theta_n$ differ in the community assignments of $|m_{\theta_n}-m_{\eta_n}|$ vertices, there are $|m_{\theta_n}-m_{\eta_n}|(n-|m_{\theta_n}-m_{\eta_n}|)$ edges that belong to either \(D_{1}(\theta_n,\eta_n)\) or \(D_{2}(\theta_n,\eta_n)\), establishing inequality (\ref{eq:B-lowerbound}) (see appendix~\ref{sub:lowerboundsDs}). According to lemma~2.2 in \citep{Kleijn21} (with $B_n=\{\theta_n\}$), for any test $\phi:\scrX_n\to[0,1]$, we have,
\[
  P_{\theta_n}\Pi(S_n|X^n)
    \leq P_{\theta_n}\phi(X^n) + \frac{1}{\pi_n(\theta_n)}
      \sum_{\eta_n\in S_n}\pi_n(\eta_n)P_{\eta_n}(1-\phi(X^n)).
\]
Based on lemma~2.7 in \citep{Kleijn21}, lemma~\ref{lem:testingpower} proves that for any $\eta_n\in S_n$ there is a test function $\phi_{\eta_n}$ that distinguishes $\theta_n$ from $\eta_n$ as follows,
\[
  P_{\theta_n}\phi_{\eta_n}(X^n)
    + \frac{\pi_n(\eta_n)}{\pi_n(\theta_n)} P_{\eta_n}(1-\phi_{\eta_n}(X^n))
  \leq \frac{\pi_n(\eta_n)^{1/2}}{\pi_n(\theta_n)^{1/2}}\rho(p_n,q_n)^{d_n},
\]
where the last inequality follows from \(\rho(p_n, q_n)\leq 1\) and the fact that \(|D_{1}(\theta_n,\eta_n)\cup D_{2}(\theta_n,\eta_n)|\geq d_n\), for all \(\eta_n\in S\). Then, using test functions $\phi_{S_n}(X^n)=\max\{\phi_{\eta_n}(X^n):\eta_n\in S_n\}$, we have,
\[
  P_{\theta_n}\phi_{S_n}(X^n)
    \leq\sum_{\eta_n\in S_n}P_{\theta_n}\phi_{\eta_n}(X^n),
\]
so that,
\[
  \begin{split}
  P_{\theta_n}\Pi(&S_n|X^n)\\
    &\leq \sum_{\eta_n\in S_n}P_{\theta_n}\phi_{\eta_n}(X^n)
      + \frac{1}{\pi_n(\theta_n)} \sum_{\eta_n\in S_n}\pi_n(\eta_n)
      P_{\eta_n}\bigl(1-\phi_{S_n}(X^n)\bigr)\\[.5mm]
    &\leq \sum_{\eta_n\in S_n} \Bigl(P_{\theta_n}\phi_{\eta_n}(X^n)
      + \frac{\pi_n(\eta_n)}{\pi_n(\theta_n)}
      P_{\eta_n}\bigl(1-\phi_{\eta_n}(X^n)\bigr)\Bigr)\\
    &\leq \rho(p_n,q_n)^{d_n}\sum_{\eta_n\in S_n}
      \frac{\pi_n(\eta_n)^{1/2}}{\pi_n(\theta_n)^{1/2}}.
  \end{split}
\]
\end{proof}
Note that non-uniform priors $\Pi_n$ (\eg\ sample first a smallest community size $m$ (uniformly, binomially, \etc) and then $\theta_n|m$ (uniformly) from $\Theta_{n,m}$) do not help in inequality~(\ref{eq:ordervstestpwr}): because $\theta_n$ is unknown, the factor $\pi_n(\theta_n)^{-1/2}$ can only be dominated by $\inf\{\pi_n(\theta):\theta\in\Theta_n\}^{-1/2}$. For most priors this leads to exponential factors of the type $\exp(ng)$ with a prior-dependent constant $g>0$ \citep{Waaij21}, while in the uniform case, the upper bound of inequality~(\ref{eq:ordervstestpwr}) matches pointwise testing power $\rho(p_n,q_n)^{d_n}$ strictly versus the cardinal $|S_n|$. As a consequence, all convergence results in the next section are optimal for priors $\Pi_n$ that are uniform on $\Theta_n$, and we do not consider non-uniform priors from this point onward.

%%%%%%%%%%%%%%%%%%%%%%%%%%%%%%%%%%%%%%%%%%%%%%%%%%%%%%%%%%%%%%%%%%%%%%%%%%%%%%%

\section{Recovery of community assignments}
\label{sec:commrecovery}

When a statistical model has a natural partition into a finite number of submodels (like the size of the smallest community in the current model), the question arises whether it is possible to \emph{first} select one of the sub-models, and then restrict estimation within that sub-model. Such a procedure can lead to significant reduction in complexity of the estimation procedure (and of the computational burden); if model selection can be done consistently, the benefits are often great. So before we commit to recovery of the full community structure, we should explore the possibility of first model-selecting the smallest community size. This analysis has been done in detail and can be found in \citep{Kleijn22}. The answer is that there are no short-cuts: consistent selection of the smallest community size \emph{without} also addressing the estimation question is not feasible in a straightforward manner. Hence, we analyse the question of community recovery without the benefit of consistent model selection for the unknown size of the smallest community. In subsections~\ref{sub:exactrecovery} and~\ref{sub:almostexactrecovery} we discuss posterior concentration on and around the true community assignment vectors $\theta_n$.

\subsection{Exact recovery of the community structure} 
\label{sub:exactrecovery}

\begin{theorem}
\label{thm:exactrecovery}
For fixed $n\geq1$, suppose $X^n$ is generated according to $P_{\theta_{n}}$ with $\theta_{n}\in\Theta_n$ and choose the uniform prior on $\Theta_n$. Then,
\begin{equation}
  \label{eq:exactineq}
  P_{\theta_{n}}\Pi\bigl(\,\{\theta_{n}\}\bigm| X^n\bigr)
  \geq 1- \frac{n}2\rho(p_n,q_n)^{n/2}\,e^{n\rho(p_n,q_n)^{n/2}},
\end{equation}
implying that if,
\begin{equation}
  \label{eq:condforexactrecovery}
  n\rho(p_n,q_n)^{n/2}\to0, 
\end{equation}
then the posterior recovers the true community assignment exactly.
\end{theorem}
\begin{proof}
For any integer $k\geq0$, define $V_{n,k}(\theta_n)=\set{\eta_n\in\Theta_{n}:k(\theta_n,\eta_n)=k}$. Note that for $k=1,\ldots,\floor{n/2}$, $V_{n,k}$ has at most $\binom{n}{k}$ elements and, when $n$ is even, $V_{n,n/2}$ has at most $1/2\binom{n}{n/2}$ elements. It follows from equation~(\ref{eq:sizeofD1andD2inVnmk}) that for all $\eta_n\in V_{n,k}$, $|D_1(\theta_n,\eta_n)\cup D_2(\theta_n,\eta_n)|=k(n-k)$. Then proposition~\ref{prop:postconvset} (with uniform prior) says that,
\[
  \begin{split}
    P_{\theta_{0,n}}\Pi(&\Theta_{n}\setminus \set{\theta_{0,n}}|X^n)
      =\sum_{k=1}^{\floor{n/2}} P_{\theta_{0,n}}\Pi(V_{n,k}(\theta)\mid X^n)\\
  &\leq \frac12\sum_{k=1}^{n}\binom{n}{k}\rho(p_n,q_n)^{k(n-k)}
    \leq \frac{n}2\rho(p_n,q_n)^{n/2}\,e^{n\rho(p_n,q_n)^{n/2}},
  \end{split}
\]
where we use lemma~\ref{lem:boundforbinomialsum} for the second bound.
\end{proof}
In the following corollary, we explore the condition of theorem~\ref{thm:exactrecovery} more closely in the Chernoff-Hellinger phase.
\begin{corollary}
\label{cor:mnscritical}
Assume the conditions of theorem~\ref{thm:exactrecovery}. If the sequences $a_n,b_n$ in the Chernoff-Hellinger phase satisfy,
\begin{equation}
  \label{eq:newcritical}
  \Bigl((\sqrt{a_n}-\sqrt{b_n})^2-\frac{a_nb_n\log(n)}{2n}-4\Bigr)\log(n)\to \infty,
\end{equation}
then the posterior recovers the community assignments exactly.
\end{corollary}
\begin{proof}
Since for all $x\in[0,1]$, $\sqrt{1-x}\leq 1-x/2$,
\[
  \begin{split}
  \rho(p_n,q_n) &\leq \sqrt{p_nq_n} + (1-p_n/2)(1-q_n/2)
  = 1  - \ft12(\sqrt{p_n}-\sqrt{q_n})^2 + \ft14p_nq_n \\ 
  &= 1 - \frac1n\Bigl(\ft12(\sqrt{a_n}-\sqrt{b_n})^2\log n
    - \frac{a_nb_n}{4n}(\log n)^2\Bigr).
  \end{split}
\]
It follows that, 
\[
\begin{split}
  n\rho(p_n,q_n)^{n/2}
  \leq \exp\Bigl(\bigl(1-\ft14(\sqrt{a_n}-\sqrt{b_n})^2\bigr)\log n
    + \frac{a_nb_n}{8n}(\log n)^2\Bigr),
  \end{split}
\]
from lemma~\ref{lem:oneplusxdivrtothepowerrissmallerthanetothepowerx}.
\end{proof}
Note that condition (\ref{eq:newcritical}) resembles (but is not exactly equal to) (\ref{eq:mnscritical}), the requirement of \citep{Mossel16}, which applies only if there exists a constant $C>0$ such that $C^{-1}\leq a_n, b_n \leq C$ for large enough $n$ \citep{Mossel16,Zhang16}. For $a_n,b_n$ of order $O(1)$, a simple sufficient conditions for exact recovery is,
\begin{equation}
  \label{eq:ourMNSdetect}
  \bigl((\sqrt{a_n}-\sqrt{b_n})^2-4\bigr)\log n\to\infty,
\end{equation}
which does not require that $a_n,b_n$ stay bounded away from $0$. Note: if we disregard the (negligible) term proportional to $\log(\log(n))$ in (\ref{eq:mnscritical}), there is a relative factor two between the lower-bounding constants of conditions~(\ref{eq:ourMNSdetect}) and~(\ref{eq:mnscritical}) (possibly a manifestation of the fact that the smallest community size is not half of $n$ but unknown).

\begin{example}
\label{ex:qniszero}
Note that exact recovery of the community structure is not possible in the Kesten-Stigum phase. This can be understood intuitively on the basis of the special case where $q_n=0$: if $p_n$ is of order $O(n^{-1}\log(n))$, the two communities form as Erd\H os-R\'enyi graphs that are connected with a probability that goes to one as $n\to\infty$ \citep{hofstad16}, making exact recovery asymptotically trivial. If $q_n=0$ and $p_n\geq Cn^{-1}$ for some $C>1$, the two communities form as Erd\H os-R\'enyi graphs with two independent \emph{giant components} containing some non-zero fraction of all vertices asymptotically, but fragments of $O(\log(n))$ vertices remain unconnected to either \citep{hofstad16}. Consequently in the Kesten-Stigum phase exact recovery is not possible, even in the setting where $q_n=0$. The above suggests that this break-down persists in case where the edge probabilities $q_n$ are non-zero.
\end{example}

\subsection{Almost-exact recovery of the community structure} 
\label{sub:almostexactrecovery}

For block models with even higher degrees of edge sparsity, we consider the condition for almost exact recovery with posteriors. Let $(k_n)$ be a sequence with $0\leq k_n\leq\floor{n/2}$, let $\theta_n$ be community assignments in $\theta_n$. Define the (Hamming-)metric balls,
\begin{equation}
  \label{eq:definitionBkn}
  B_n(\theta_n,k_n)=\bigl\{\eta_n\in\Theta_n:k(\eta_n,\theta_n)\leq k_n\},
\end{equation}
based on definition~(\ref{eq:kmetric}). Metric balls of this type contain $\theta_n$ and all community assignments that differ by no more than $k_n$ vertices from $\theta_n$. If the posterior concentrates in the balls $B_n(\theta_n,k_n)$ with high probability, then we estimate the community assignment correctly up to subsets of vertices of order $O(k_n)$ with high probability. For instance in example~\ref{ex:qniszero}, communities manifest as giant components with unconnected fragments of order $O(\log(n))=o(n)$, so we could take $k_n$ proportional to $n$. In such cases, \emph{almost-exact recovery} (definition~\ref{def:detect}) is appropriate, and the following theorem describes the condition on edge sparsity and error rate $k_n$ that enables almost-exact recovery with posterior distributions.
\begin{theorem}
\label{thm:almostexactrecovery}
For fixed $n\geq1$, suppose $X^n$ is generated according to $P_{\theta_n}$ with $\theta_n\in\Theta_n$ and choose the uniform prior on $\Theta_n$. For some sequence $a_n$ with $0<a_n<1/2$, let $k_n$ be an integer such that $k_n\geq a_nn$. Then the expected posterior probability of $B_n(\theta_n,k_n)$ is lower bounded as follows,
\begin{equation}
  \label{eq:postconcKS}
  P_{\theta_n}\Pi\bigl(B_n(\theta_n,k_n)\bigm| X^n\bigr)
  \geq 1-\frac12\Bigl(\ft{e}{a_n}\rho(p_n,q_n)^{n/2}\Bigr)^{a_nn}
      \Bigl(1-\ft{e}{a_n}\rho(p_n,q_n)^{n/2}\Bigr)^{-1}.
\end{equation}
\end{theorem}
\begin{proof}
By proposition~\ref{prop:postconvset} (and using the sets $V_{n,k}(\theta_n)$ of the proof of theorem~\ref{eq:condforexactrecovery}), when $k_n\geq a_n n$, we see that,
\begin{equation}
  \label{eq:upperboundkerstenstigumphase}
  \begin{split}
    P_{\theta_n}&\Pi\bigl(\Theta_{n}\setminus B_n(\theta_n,k_n)\bigm| X^n\bigr)
     =\sum_{k=k_n+1}^{\floor{n/2}}P_{\theta_{0,n}}
       \Pi\bigl(V_{n,k}(\theta_n)\bigm| X^n\bigr)\\
    &\leq \frac12\sum_{k=k_n}^{n}\binom nk \rho(p_n,q_n)^{kn/2}
    \leq \frac12\sum_{k=k_n}^{n} \Bigl(\frac{en}k\Bigr)^k \rho(p_n,q_n)^{kn/2}\\
    &\leq \frac12\sum_{k=k_n}^{\infty} \Bigl(\frac{e}{a_n}\Bigr)^k
      \rho(p_n,q_n)^{kn/2}\\
    &\leq \frac12\Bigl(\ft{e}{a_n}\rho(p_n,q_n)^{n/2}\Bigr)^{a_nn}
      \Bigl(1-\ft{e}{a_n}\rho(p_n,q_n)^{n/2}\Bigr)^{-1},
  \end{split}
\end{equation}
proving the assertion.
\end{proof}
Almost exact recovery is established when $P_{\theta_n}\Pi(\,\Theta_{n}\setminus B_n(\theta_n,k_n)| X^n) $ converges to zero (possibly while $a_n\downarrow0$). As in example~\ref{ex:qniszero} almost-exact recovery is especially relevant in the Kesten-Stigum phase, which we consider separately in the following proposition.
\begin{proposition}
\label{prop:KSalmostexact}
Assume the conditions of theorem~\ref{thm:almostexactrecovery}. If the sequences $c_n,d_n$ in the Kesten-Stigum phase and the fractions $a_n$ satisfy,
\begin{equation}
  \label{eq:almostexactexponent}
  a_nn\Bigl(\log(a_n)+\frac14\bigl(\sqrt{c_n}
    -\sqrt{d_n}\bigr)^2-1\Bigr)\to\infty
\end{equation}
then posteriors recover the community assignment almost-exactly with any error rate $k_n\geq a_nn$.
\end{proposition}
\begin{proof}
Again using that for all $x\in[0,1]$, $\sqrt{1-x}\leq 1-x/2$, we find,
\[
  \rho(p_n,q_n) 
  \leq \frac{\sqrt{c_nd_n}}{n}
    +\Bigl(1-\frac{c_n}n\Bigr)\Bigl(1-\frac{d_n}n\Bigr)
  \leq 1-\frac{\bigl(\sqrt{c_n}-\sqrt{d_n}\bigr)^2}{2n}+\frac{c_nd_n}{4n^2},
\]
and using lemma~\ref{lem:oneplusxdivrtothepowerrissmallerthanetothepowerx},
\[
  \frac{e}{a_n}\rho(p_n,q_n)^{n/2}
  \leq \exp\Bigl(1-\log(a_n)
    -\frac{\bigl(\sqrt{c_n}-\sqrt{d_n}\bigr)^2}{4}+\frac{c_nd_n}{8n}\Bigr).
\]
Based on (\ref{eq:upperboundkerstenstigumphase}), we arrive at posterior concentration in the sets $B_n(\theta_n,k_n)$ if,
\[
a_nn\Bigl(\log(a_n)+\frac14\bigl(\sqrt{c_n}
    -\sqrt{d_n}\bigr)^2-\frac1{8n}c_nd_n-1\Bigr)\to\infty.
\]
Since $c_n,d_n$ are of order $o(\log(n))$, the third term is negligible and we conclude that posterior concentration occurs whenever (\ref{eq:almostexactexponent}) holds.
\end{proof}
Let us illustrate how requirement (\ref{eq:almostexactexponent}) relates to condition~(\ref{eq:MNSdetect}) and the criteria of \citep{Decelle11a,Decelle11b}. In sparse situations where $p_n,q_n=o(1)$, we can expand the function $p\mapsto\sqrt{p}$ around the value $\ft12(p_n+q_n)$, for every $n\geq1$, to obtain,
\[
  \sqrt{p_n}-\sqrt{q_n} =
  \frac1{2\sqrt{\frac12(p_n+q_n)}}(p_n-q_n)+O(|p_n-q_n|^2).
\]
which implies that,
\[
  \bigl(\sqrt{c_n}-\sqrt{d_n}\bigr)^2
  =\frac{(c_n-d_n)^2}{2(c_n+d_n)} + O(n^{-1}),
\]
in terms of the sequences $(c_n)$, $(d_n)$. This means that $(\sqrt{c_n}-\sqrt{d_n})^2\to\infty$ is equivalent to equation~(\ref{eq:MNSdetect}). Based on that observation, we discuss the consequences of proposition~\ref{prop:KSalmostexact} in several specific corollaries.

In case we allow for error rates $k_n=a_nn$ that leave a non-zero fraction of mis-assigned vertices in the limit ($0<a=\liminf_na_n<1/2$), we find the following simple sufficient condition of the form of condition~(\ref{eq:decellescondition}), conjectured by \cite{Decelle11a,Decelle11b}:
\begin{corollary}
\label{cor:decellescondition}
Assume the conditions of theorem~\ref{thm:almostexactrecovery}, and let $0<a<1/2$ be given. If, for some constant $C>1$ and large enough $n$,
\begin{equation}
  \label{eq:weakestconditionforalmostexactrecovery}
  \bigl(\sqrt{c_n}-\sqrt{d_n}\bigr)^2 > 4C\bigl(1-\log(a)\bigr),
\end{equation}
then the posterior recovers the true community assignment almost exactly with error rate $k_n=an$. \end{corollary} Comparing condition~(\ref{eq:weakestconditionforalmostexactrecovery}) with condition~(\ref{eq:decellescondition}), a relative factor four appears in the lower bound due to the unknown smallest community size, as well as a $\log(a)$-proportional correction term that raises the lower-bounding constant further. Condition~(\ref{eq:MNSdetect}) implies (\ref{eq:weakestconditionforalmostexactrecovery}) but not the other way around. Indeed, according to (\ref{eq:almostexactexponent}) above, condition (\ref{eq:MNSdetect}) is sufficient for almost exact posterior recovery with \emph{any} fixed rate $k_n=an$, $0<a<1/2$, which implies what is called \emph{weak consistency} in \cite{Mossel16}.
\begin{corollary}
\label{cor:MNSdetect}
Assume the conditions of theorem~\ref{thm:almostexactrecovery}. If condition~(\ref{eq:MNSdetect}) holds, the posterior recovers the true community assignment almost exactly with error rate $k_n=a_nn$ for \emph{some vanishing fraction} $a_n\to0$.
\end{corollary}
In cases where $\liminf_na_n=0$, the rate at which $a_n$ decreases to zero is to be compensated in (\ref{eq:almostexactexponent}) by faster divergence of the limit (\ref{eq:MNSdetect}).
\begin{corollary}
Assume the conditions of theorem~\ref{thm:almostexactrecovery} and let $0<a_n<1/2$ be given, such that $a_n\to0$, $a_nn\to\infty$. If, for some constant $C>1$ and large enough $n$,
\begin{equation}
  \label{eq:edgessparsityandvanishingfraction}
  (\sqrt{c_n}-\sqrt{d_n})^2 + 4C\log(a_n)\to\infty,
\end{equation}
then the posterior recovers the community assignments almost exactly with error rate $k_n=a_nn$.
\end{corollary}
\begin{example}
\label{ex:lognerrorrate}
For an extreme example of the latter kind, consider error rates of order $O(\log(n))$, \eg\ with fractions $a_n$ of order $O(\log(n)/n)$, condition~(\ref{eq:edgessparsityandvanishingfraction}) reads,
\[
  (\sqrt{c_n}-\sqrt{d_n})^2-4C\log(n)\to\infty,
\]
(up to a $\log(\log(n))$-term) for some constant $C>1$ and large enough $n$, forcing edge sparsity up to the $\log(n)/n$-level that characterizes the Chernoff-Hellinger phase. Comparison with condition~(\ref{eq:ourMNSdetect}) then leads us to conclude that in any situation where almost-exact recovery with error rates as small as $O(\log(n))$ is possible, the posterior recovers the true community assignment exactly. This is possibly related to the fact that fragments unconnected to the giant component in the Erd\H os-R\'enyi graph, are at most of order $O(\log(n))$ with high probability (see \citep{hofstad16} and example~\ref{ex:qniszero}).
\end{example}

\section{Uncertainty quantification}
\label{sec:pbmuncertainty}

As said in the introduction, approximation or simulation of a posterior distribution is computationally costly, and if the statistical goal is only the estimation of the community assignment, more efficient algorithms are known, also under edge sparsity (see \citep{Abbe18} for an overview). When more complex statistical questions like uncertainty quantification and hypothesis testing are the goal, sampling distributions for said algorithms are required and those are often prohibitively hard to obtain. In this section we show that enlargement of Bayesian credible sets offers a viable alternative, with finite amounts of data. Enlargements of credible sets also feature centrally in asymptotic conversion of credible sets to confidence sets as in \citep{Kleijn21}.

Let us first fix the relevant definitions. Bayesian uncertainty quantification relies on the notion of credibility.
\begin{definition}
Given $n\geq1$, a prior $\Pi_n$, $0\leq\gamma<1$ and data $X^n$, a \emph{credible set} of \emph{credible level} $1-\gamma$ is any subset $D(X^n)\subset\Theta_n$ that receives posterior mass at least $1-\gamma$:
\[
  \Pi\bigl( D(X^n)\bigm|X^n\bigr)\geq 1-\gamma,
\]
$P^{\Pi_n}$-almost-surely (see definitions~\ref{def:priorpredictive} and~\ref{def:posterior}). In case $\gamma=0$, $D(X^n)$ is the support of the posterior.
\end{definition}
(The notation for credible sets involves $X^n$ to emphasize that credible sets are constructed from the posterior, and hence, depend on the data $X^n$.) The most natural way to compile a credible set $D(X^n)$ in a discrete space like $\Theta_n$, is to calculate the posterior weights $\Pi(\{\theta\}|X^n)$ of all $\theta\in\Theta_n$, order the $\theta_n$ by decreasing posterior weight into a finite sequence $\theta_{n,1}(X^n)$, $\theta_{n,2}(X^n)$, $\ldots$, $\theta_{n,|\Theta_n|}(X^n)$, and define $D(X^n)=\{\theta_{n,1}(X^n),\ldots,\theta_{n,m}(X^n)\}$, for the smallest $m\geq1$ such that $\Pi(D(X^n)|X^n)$ is greater than or equal to the required credible level. Note that $\theta_{n,1}(X^n)$ is the \emph{maximum-a-posteriori}-estimator (which, in the case of a uniform prior, is equal to the maximum-likelihood estimator).

Similarly, the frequentist uses the notion of confidence for uncertainty quantification. 
\begin{definition}
Given an unknown $\theta_n\in\Theta_n$ and an observation $X^n\sim P_{\theta_n}$, a \emph{confidence set} $C(X^n)\subset\Theta_n$ of \emph{confidence level} $1-\alpha$, $(0<\alpha<1)$, is defined by any ($\theta_n$-independent) set-valued map $x^n\mapsto C(x^n)\subset\Theta_n$ such that,
\[
  P_{\theta_n}\bigl(\theta_n\in C(X^n)\bigr) \geq 1-\alpha.
\]
\end{definition}

In the Chernoff-Hellinger phase with a posterior that succeeds in exact recovery, all posterior mass ends up in the singleton $\{\theta_n\}$ containing the true community assignment with high probability, so it is clear that \emph{any} sequence of credible sets $D_n(X^n)$ of credible levels $1-\gamma_n$ with $\liminf_n\gamma_n>0$, will contain $\theta_n$ with high $P_{\theta_n}$-probability as $n\to\infty$. Because of theorems~\ref{thm:exactrecovery} and~\ref{thm:almostexactrecovery}, we can consider a version of this argument that holds in full generality at finite graphs size $n$.
\begin{lemma}
\label{lem:crediblesettoconfidencesets}
Fix $n\geq1$ and some prior $\Pi_n$ on $\Theta_n$, let $\theta_n\in\Theta_n$ and $X^n\sim P_{\theta_n}$ be given. Let $B\subset\Theta_n$ be a subset with expected posterior probability that is lower-bounded,
\begin{equation}
  \label{eq:posteriormassinB}
  P_{\theta_n}\Pi\bigl(\, B \bigm| X^n \bigr)\geq
  1-\beta,
\end{equation}
for some $0<\beta<1$. For any $0<\gamma<1$ and any credible set $D(X^n)\subset\Theta_n$ of level $1-\gamma$,
\[
  P_{\theta_n}\bigl(B\cap D(X^n)\neq\emptyset\bigr)
  \geq 1-\frac{\beta}{1-\gamma}.
\] 
\end{lemma}
\begin{proof}  
We first prove that for every $0<r<1$,
\[
  P_{\theta_n}\bigl(\Pi(B| X^n)\geq r\bigr)
  \geq 1-\frac{\beta}{1-r},
\]
by contradiction: let $\delta>0$ be given and define the event,
\[
  E=\bigl\{\,x_n\in\scrX_n\,:\,
    \Pi\bigl(B\bigm| X^n=x^n\bigr)\geq r\,\bigr\}.
\]
Suppose that $P_{\theta_n}(E)\leq 1-\beta/(1-r)-\delta$. Then,
\begin{equation}
  P_{\theta_n}\Pi(B| X^n)
  \leq P_{\theta_n}(E) + r(1-P_{\theta_n}(E))
  \leq 1-\beta-\delta (1-r)<1-\beta,
\end{equation}
which contradicts the assumption that $P_{\theta_n}\Pi(B| X^n)\geq 1-\beta$. Since this holds for every $\delta>0$, we have $P_{\theta_n}(E)\geq 1-\beta/(1-r)$. Choose $r>\gamma$. As $D(X^n)$ has posterior mass of at least $1-\gamma$, $B$ and $D(x^n)$ cannot be disjoint for $x^n\in E$. So,
\[
  P_{\theta_n}\bigl(B\cap D(X^n)\neq\emptyset\bigr)
  \geq P_{\theta_n}(E)\geq 1-\frac\beta{1-\gamma},
\]
which proves the assertion.
\end{proof}
Based on the Bernstein-von~Mises theorem \citep{LeCam90} and other arguments \citep{Ghosal17,Kleijn21}), one might expect the relation between Bayesian and frequentist uncertainty quantification to involve some type of proportionality between credible and confidence levels also at finite sample sizes. Somewhat surprisingly, it emerges that the finite-sample confidence level of a credible set depends mostly on the expected amount of mis-placed posterior probability and less on the credible level.

Under the conditions of theorem~\ref{thm:exactrecovery}, condition~(\ref{eq:posteriormassinB}) holds with $\rho(p_n,q_n)$-dependent $\beta$. We record the conclusion in the form of the following proposition.
\begin{proposition}
\label{prop:exactcredconf}% was {cor:exactcredconf}
For fixed $n\geq1$, suppose $X^n$ is generated according to $P_{\theta_n}$ with $\theta_n\in\Theta_n$ and choose the uniform prior on $\Theta_n$. Every credible set $D(X^n)$ of credible level $1-\gamma$ is a confidence set of confidence level,
\begin{equation}
  \label{eq:boundexactcredconf}
  P_{\theta_n}\bigl(\theta_n\in D(X^n)\bigr)
  \geq 1-\frac{n}{2(1-\gamma)}
    \rho(p_n,q_n)^{n/2}\,e^{n\rho(p_n,q_n)^{n/2}}.
\end{equation}
\end{proposition}
\begin{proof}
Choose $B=\{\theta_n\}$ in lemma~\ref{lem:crediblesettoconfidencesets} and use theorem~\ref{thm:exactrecovery}.
\end{proof}
To use proposition~\ref{prop:exactcredconf} for the construction of confidence sets, one takes the following steps: practical situations involve some given graph size $n\geq1$, known edge probabilities $p_n=p$, $q_n=q$ and a realised graph $X^n=x^n$, with associated realised posterior $\Pi(\,\cdot\,|X^n=x^n)$. Given a desired confidence level $0<1-\alpha<1$, we choose credible level,
\begin{equation}
  \label{eq:credlevelexactcredconf}
  1-\gamma =\min\bigl\{1, 
  (n/2\alpha)\rho(p,q)^{n/2}\,e^{n\rho(p,q)^{n/2}}\bigr\}.
\end{equation}
With large $n$, $(n/2)\rho(p,q)^{n/2}$ is small and $1-\gamma$ lies below one for large enough graph size. We then interpret any realised credible set $D(x^n)$ of credible level $1-\gamma$ as a confidence set of level $1-\alpha$. Note that as $n$ grows or $p$ and $q$ are further apart, the credible level $1-\gamma$ is closer to zero, making the corresponding credible sets smaller.
\begin{example}
\label{ex:thirtyvertices}
With a graph containing $n=25$ vertices, edge probabilities $p=0.9$, $q=0.1$ and a desired confidence level $1-\alpha=0.95$, $\rho(p,q)=0.6$ and $(n/2)\rho(p,q)^{n/2}\approx0.0211$, so that any credible set of credible level $1-\gamma\approx0.422$ is also a confidence set of confidence level $0.95$. Keeping $p,q$ fixed, the dependence on $n$ is quite sensitive and changes sharply around the point $n=25$: for graph sizes below $n=25$, $1-\gamma$ is (close to) one (and we need to include all or most of the points that receive non-zero posterior mass in the credible set); for graph sizes (well) above $n=25$, credible levels $1-\gamma$ close to $0$ are good enough (and we need to include only a relatively small set of points with the highest amounts of posterior probability in the credible set).
\begin{figure}[htb]
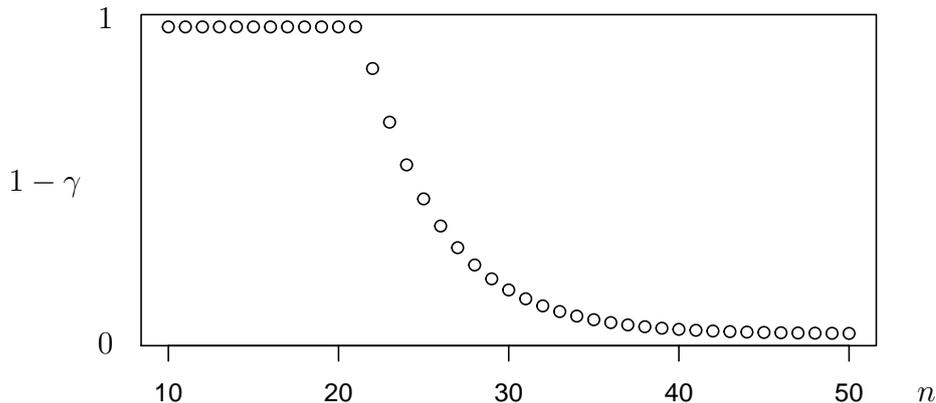

  \begin{center}
    \parbox{0.9\textwidth}{%\hspace*{-5ex}
      \begin{lpic}{crediblelevel(0.8)}
        \lbl[t]{15,82;{$1$}}
        \lbl[t]{5,55;{$1-\gamma$}}
        \lbl[t]{15,28;{$0$}}
        \lbl[t]{150,19;{$n$}}
      \end{lpic}\vspace*{-3em}
      \caption{\label{fig:crediblelevel} Credible level $1-\gamma$
      required for a confidence set of confidence level
      $1-\alpha=0.95$, as a function of graph size $n$, with fixed
      edge probabilities $p=0.9$ and $q=0.1$. There is a sharp
      decrease in required credible level around graph size
      $n=25$, indicating that the frequentist has confidence
      in community assignments of high posterior probability
      rather than in subsets of almost full posterior probability.
      In this case, the critical graph size $n(0.9,0.1;0.05)=25$.}
    }
  \end{center}
\end{figure}
At intermediate values of $n$ where $1-\gamma$ is changing from one to zero, the frequentist decides to have confidence not just in subsets of almost full posterior probability, but also in sets of smaller posterior probability, because he knowns that for large-enough graph sizes, the posterior has concentrated far enough.
\end{example}
\begin{remark}
\label{rem:ruleofthumb}
The conclusion of the previous example can also be given the following form: given a desired confidence level $1-\alpha$ and edge probabilities $p,q$, there exists a \emph{critical graph size},
\begin{equation}
  \label{eq:critgraphsize}
  n(p,q;\alpha) = \min\bigl\{ n\,:\,
    n\rho(p,q)^{n/2}\,e^{n\rho(p,q)^{n/2}}<\alpha \bigr\},
\end{equation}
where the frequentist first uses credible sets of credible level below $1/2$ as confidence sets of level $1-\alpha$. If the graph size lies (well) above $n(p,q;\alpha)$, very small credible sets (containing only the maximum-a-posteriori/maximum-likelihood estimator and a relatively small number of other community assignments of high posterior probability) are confidence sets of level $1-\alpha$; if the graph size lies below $n(p,q;\alpha)$, (most of) the support of the posterior is required to form a confidence set of level $1-\alpha$.
\end{remark}

Under the conditions of theorem~\ref{thm:almostexactrecovery}, credible sets have to be enlarged to satisfy condition~(\ref{eq:posteriormassinB}): for any credible set $D(X^n)$ and a non-negative integer $k$, we define the $k$-enlargement $C(X^n)$ of $D(X^n)$ to be the union of all Hamming balls of radius $k\geq1$ that are centred on points in $D(X^n)$,
\[
  C(X^n)=\bigl\{\theta_n\in\Theta_n:\exists_{\eta_n\in D_n(X^n)},
  k(\theta_n,\eta_n)\leq k\bigr\}.
\]
In the argument leading to proposition~\ref{prop:exactcredconf}, we only have to replace the singleton $\{\theta_n\}$ with a (Hamming-)ball $B_n(\theta_n,k)$ (see definition~(\ref{eq:definitionBkn})): according to lemma~\ref{lem:crediblesettoconfidencesets}, if $B_n(\theta_n,k)$ receives mass $1-\beta$, then the radius-$k$ enlargement of any credible set of level $1-\gamma$ is a confidence set of level $1-\beta(1-\gamma)^{-1}$.
\begin{proposition}
\label{prop:almostexactcredconf}% was {cor:almostexactcredconf}
For fixed $n\geq1$, suppose $X^n$ is generated according to $P_{\theta_n}$ with $\theta_n\in\Theta_n$ and choose the uniform prior on $\Theta_n$. For given $0<a<1/2$, define $k=\ceiling{a n}$. Then the $k$-enlargement $C(X^n)$ of a credible set $D(X^n)$ of level $1-\gamma$ is a confidence set of confidence level,
\begin{equation}
  \label{eq:boundalmostexactcredconf}
  P_{\theta_n}\bigl(\theta_n\in C(X^n)\bigr)
  \geq 1- \frac{1}{2(1-\gamma)}
    \Bigl(\ft{e}{a}\rho(p_n,q_n)^{n/2}\Bigr)^{an}
      \Bigl(1-\ft{e}{a}\rho(p_n,q_n)^{n/2}\Bigr)^{-1}.
\end{equation}
\end{proposition}
\begin{proof}
Choose $B=B_n(\theta_{0,n},k)$ in lemma~\ref{lem:crediblesettoconfidencesets} and use equation~(\ref{eq:postconcKS}).
\end{proof}
Proposition~\ref{prop:almostexactcredconf} is used as follows: assume we have a
realised graph $X^n=x^n$ and known edge probabilities $p_n=p$, $q_n=q$. Denote the associated realised posterior by $\Pi(\,\cdot\,|X^n=x^n)$. For any $a>0$ and any desired confidence level $0<1-\alpha<1$, we choose credible level,
\begin{equation}
  \label{eq:credlevelalmostexactcredconf}
  1-\gamma =\min\biggl\{1, 
    \frac1{2\alpha}\Bigl(\ft{e}{a}\rho(p,q)^{n/2}\Bigr)^{an}
      \Bigl(1-\ft{e}{a}\rho(p,q)^{n/2}\Bigr)^{-1}
      \biggr\}.
\end{equation}
This expression suggests that error fractions $a$ roughly of order $\rho(p,q)^{n/2}$ are the most appropriate. For large enough $n$, $1-\gamma$ lies below one and we interpret the $\ceiling{an}$-enlargement $C(x^n)$ of any realised credible set $D(x^n)$ of credible level $1-\gamma$ as a confidence set of level $1-\alpha$.
\begin{example}
\label{ex:almostthirtyvertices}
Again we consider a graph with $n=25$ vertices, edge probabilities $p=0.9$, $q=0.1$ and a desired confidence level $1-\alpha=0.95$, $\rho(p,q)=0.6$. For $a=0.05$, $0.1$ or $0.25$ (which would allow for fixed $5\%$, $10\%$ or $25\%$ fractions of mis-assigned vertices in the Hamming balls of theorem~\ref{thm:almostexactrecovery}), we plot the required credible levels in figures~\ref{fig:almostcrediblelevel005}--\ref{fig:almostcrediblelevel025}.
\begin{figure}[htb]
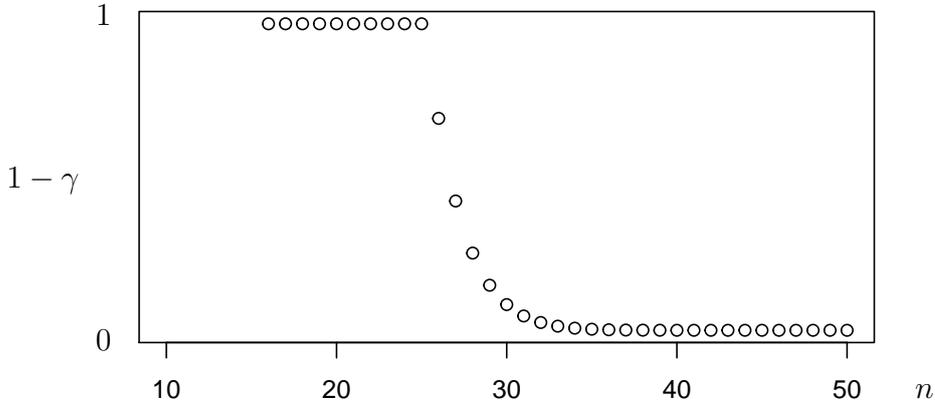

  \begin{center}
    \parbox{0.9\textwidth}{%\hspace*{-5ex}
      \begin{lpic}{enlarged005crediblelevel(0.8)}
        \lbl[t]{15,82;{$1$}}
        \lbl[t]{5,55;{$1-\gamma$}}
        \lbl[t]{15,28;{$0$}}
        \lbl[t]{150,19;{$n$}}
      \end{lpic}\vspace*{-3em}
      \caption{\label{fig:almostcrediblelevel005} Credible level $1-\gamma$
      required for a confidence set of confidence level
      $1-\alpha=0.95$ and Hamming enlargement radius
      $k=\ceiling{0.05n}$, as a function of graph size $n$,
      with fixed edge probabilities $p=0.9$ and $q=0.1$. Note
      the decrease in required credible level around the
      critical graph size $n(0.9,0.1;0.05,0.05)=27$.}
    }
  \end{center}
\end{figure}
\begin{figure}[htb]
  \begin{center}
    \parbox{0.9\textwidth}{%\hspace*{-5ex}
      \begin{lpic}{enlarged010crediblelevel(0.8)}
        \lbl[t]{15,82;{$1$}}
        \lbl[t]{5,55;{$1-\gamma$}}
        \lbl[t]{15,28;{$0$}}
        \lbl[t]{150,19;{$n$}}
      \end{lpic}\vspace*{-3em}
      \caption{\label{fig:almostcrediblelevel010} Credible level $1-\gamma$
      required for a confidence set of confidence level
      $1-\alpha=0.95$ and Hamming enlargement radius
      $k=\ceiling{0.1n}$, as a function of graph size $n$,
      with fixed edge probabilities $p=0.9$ and $q=0.1$. Note
      the decrease in required credible level around the
      critical graph size $n(0.9,0.1;0.05,0.1)=21$.}
    }
  \end{center}
\end{figure}
\begin{figure}[htb]
  \begin{center}
    \parbox{0.9\textwidth}{%\hspace*{-5ex}
      \begin{lpic}{enlarged025crediblelevel(0.8)}
        \lbl[t]{15,82;{$1$}}
        \lbl[t]{5,55;{$1-\gamma$}}
        \lbl[t]{15,28;{$0$}}
        \lbl[t]{150,19;{$n$}}
      \end{lpic}\vspace*{-3em}
      \caption{\label{fig:almostcrediblelevel025} Credible level $1-\gamma$
      required for a confidence set of confidence level
      $1-\alpha=0.95$ and Hamming enlargement radius
      $k=\ceiling{0.25n}$, as a function of graph size $n$,
      with fixed edge probabilities $p=0.9$ and $q=0.1$. Note
      the decrease in required credible level around the
      critical graph size $n(0.9,0.1;0.05,0.25)=14$.}
    }
  \end{center}
\end{figure}
In the Kesten-Stigum phase (\cf\ theorem~\ref{thm:almostexactrecovery}), given a desired confidence level $1-\alpha$ and edge probabilities $p,q$, there again exists a \emph{critical graph size},
\begin{equation}
  \label{eq:almostcritgraphsize}
  n(p,q;\alpha,a) = \min\biggl\{ n\,:\,
    \frac{1}{\alpha}\Bigl(\ft{e}{a}\rho(p,q)^{n/2}\Bigr)^{an}
      \Bigl(1-\ft{e}{a}\rho(p,q)^{n/2}\Bigr)^{-1}
    <\alpha \biggr\},
\end{equation}
where the frequentist first uses $\ceiling{an}$-enlarged credible sets of credible level below $1/2$ as confidence sets of level $1-\alpha$. Required credible levels depend on our parameter choices as expected: if we raise the error rate from $0.05n$ to $0.25n$, the enlargement radius of credible sets grows and the required credible level decreases accordingly.
\end{example}
\begin{remark}
\label{rem:compareexactalmostexact}
To conclude we compare the bounds of propositions~\ref{prop:exactcredconf} and~\ref{prop:almostexactcredconf}: although the asymptotic definitions of the Chernoff-Hellinger and Kesten-Stigum phases suggest that we are in one or the other phase, at finite graph sizes this is inconsequential, since both bounds~(\ref{eq:boundexactcredconf}) and~(\ref{eq:boundalmostexactcredconf}) are valid and one can either choose to use credible sets of the level required by~(\ref{eq:credlevelexactcredconf}) or $\ceiling{an}$-enlarged credible sets of the level required by~(\ref{eq:credlevelalmostexactcredconf}), whichever are the smallest. Much will depend on the graph size: if $n$ lies below the critical graph size~(\ref{eq:critgraphsize}) but above the critical graph size~(\ref{eq:almostcritgraphsize}) for some $a>0$, then $\ceiling{an}$-enlarged credible sets may be preferred.
\end{remark}

\section{Discussion}
\label{sec:discussion}

The results summarized in subsection~\ref{sub:conclusions} bear some speculation regarding further exploration.

First of all the question arises whether the sufficient conditions given in section~\ref{sec:postconcentration} are also necessary. This question is interesting in its own right, but it is also important for confidence sets: if upper bounds like~(\ref{eq:boundexactcredconf}) and~(\ref{eq:boundalmostexactcredconf}) are not sharp, lower bounds for credible levels as in~(\ref{eq:credlevelexactcredconf}), (\ref{eq:credlevelalmostexactcredconf}) become unnecessary stringent and enlargement radii become unnecessarily large. It is noted that the construction of lemma~\ref{lem:crediblesettoconfidencesets} is fully general and can also be applied in other models, \eg\ with continuous parameters. In fact, the proof of the celebrated Ghosal-Ghosh-van~der~Vaart theorem \citep{Ghosal00} ends in a statement of the form (\ref{eq:posteriormassinB}) that is almost specific enough to be useful in the present context. Methods put forth in \citep[particularly, theorem~4.2 with so-called \emph{remote contiguity} as in definition~3.4]{Kleijn21} can be used directly.

Regarding uncertainty quantification in the stochastic block model, the regime where $n$ is large enough to require only small amounts of Bayesian credibility for a desired confidence level is most interesting. The space of community assignments $\Theta_n$ has cardinal $2^{n-1}$, so for large graph sizes $n$, MCMC-type samples are likely too small to properly represent the full posterior distribution. Those small samples tend to under-represent mostly the tails and not so much the bulk of the probability mass. When integrals with respect to the posterior are of interest (\eg\ the posterior mean or other minimizers of Bayesian risk functions), the tails are crucial in the calculation. But, since only community assignments with relatively high posterior probabilities are required in credible sets of low credible level, small MCMC samples may not hamper the construction of confidence sets to the same extent. This leads to the speculation that some form of \emph{early stopping} of the MCMC sequence may be justified, to enable the analysis of confidence sets not just for graph sizes where simulation of the full posterior is realistic, but possibly also for graph sizes that are (much?) larger. A numerical study could be based on cross validation of confidence levels for simulated stochastic block graphs of various sizes, to find out exactly how early one can stop the MCMC sequence.

Indeed for large values of $n$, posterior mass is concentrated almost entirely in the maximum-a-posteriori estimator (\cf\ theorem~\ref{thm:exactrecovery}) (or in Hamming balls of radii $\ceiling{an}$ surrounding the maximum-a-posteriori estimator (\cf\ theorem~\ref{thm:almostexactrecovery})), while the required credible level is low enough to let the singleton of the maximum-a-posteriori estimator (or the corresponding Hamming ball) be a valid confidence set of the desired confidence level. That perspective explains the connection with asymptotic correspondences between credible and confidence sets \citep{Kleijn18,Kleijn21}, and it would simplify the very-large-graph version of the above identification to a search for the maximum-a-posteriori estimator and a suitable choice for the error rate $a$.

%%%%%%%%%%%%%%%%%%%%%%%%%%%%%%%%%%%%%%%%%%%%%%%%%%%%%%%%%%%%%%%%%%%%%%%%%%%%%%%

\appendix

%%%%%%%%%%%%%%%%%%%%%%%%%%%%%%%%%%%%%%%%%%%%%%%%%%%%%%%%%%%%%%%%%%%%%%%%%%%%%%%

\section{Notation and conventions}
\label{app:defs}

Asymptotic statements that end in ``... with high probability''indicate that said statements are true with probabilities that grow to one as the graph size $n$ goes to infinity. The integral of a real-valued, integrable random variable $X$ with respect to a probability measure $P$ is denoted $PX$, while integrals (or, rather, sums) over the model with respect to priors and posteriors are always written out in Leibniz's or sum notation. The cardinality of a set $B$ is denoted $|B|$.

\subsection{Definitions for priors and posteriors}

For Bayesian notation, we follow \citep{Kleijn21}: assume given for every $n\geq1$, a random graph $X^n$ taking values in the (finite) space $\scrX_n$ of all undirected graphs with $n$ vertices. We denote the powerset of $\scrX_n$ by $\scrB_n$ and regard it as the domain for probability distributions $P:\scrB_n\to[0,1]$ in a model $\scrP_n$, parametrized by $\Theta_n\rightarrow\scrP_n:\theta_n\mapsto P_{\theta_n}$ with finite parameter spaces $\Theta_n$ (with powerset $\scrG_n$) and uniform priors $\Pi_n$ on $\theta_n$. As frequentists, we assume that there exists a `true, underlying distribution for the data'; in this case, that means that for every $n\geq1$, there exists a $\theta_n\in\theta_n$ and corresponding $P_{\theta_n}$ from which the $n$-th graph $X^n$ is drawn.
\begin{definition}
\label{def:priorpredictive}
Given $n\geq1$ and a prior probability measure $\Pi_n$ on $\theta_n$, define the \emph{$n$-th prior predictive distribution} as:
\begin{equation}
  \label{eq:priorpred}
  P^{\Pi_n}(X^n\in A) = \int_{\Theta_n} P_{\theta}(X^n\in A)\,d\Pi_n(\theta),
\end{equation}
for all $A\in\scrB_n$.
\end{definition}
The prior predictive distribution $P^{\Pi_n}$ is the marginal distribution for $X^n$ in the Bayesian perspective that considers parameter and sample jointly $(\theta,X^n)\in\Theta\times\scrX_n$ as the random quantity of interest.
\begin{definition}
\label{def:posterior}
Given $n\geq1$, \emph{(a version of) the posterior} is any set-function $\scrG_n\times\scrX_n\rightarrow[0,1]: (A,x^n)\mapsto\Pi(\,\theta\in A\,|X^n=x^n)$ such that,
\begin{enumerate}
\item for $B\in\scrG_n$, the map $x^n\mapsto\Pi(B|X^n=x^n)$ is $\scrB_n$-measurable,
\item for all $A\in\scrB_n$ and $V\in\scrG_n$,
  \begin{equation}
    \label{eq:disintegration}
    \int_A\Pi(\theta\in V|X^n=x^n)\,dP^{\Pi_n}(x^n) = 
    \int_V P_{\theta}(X^n\in A)\,d\Pi_n(\theta).
  \end{equation}
\end{enumerate}
\end{definition}
Bayes's Rule is expressed through equality~(\ref{eq:disintegration}) and is sometimes referred to as a `disintegration' (of the joint distribution of $(\theta,X^n)$). Because the models $\scrP_n$ are dominated (denote the density of $P_{\theta}$ by $p_{\theta}$), the fraction of integrated likelihoods,
\begin{equation}
  \label{eq:posteriorfraction}
  \Pi(\theta\in V|X^n)= 
  {\displaystyle{\int_V p_{\theta}(X^n)\,d\Pi_n(\theta)}} \biggm/
  {\displaystyle{\int_{\Theta_n} p_{\theta}(X^n)\,d\Pi_n(\theta)}},
\end{equation}
for $V\in\scrG_n$, $n\geq1$ defines a version of the posterior distribution.

%%%%%%%%%%%%%%%%%%%%%%%%%%%%%%%%%%%%%%%%%%%%%%%%%%%%%%%%%%%%%%%%%%%%%%%%%%%%%%%

\section{Tests for community assignment}
\label{app:PBMtests}

Given $n\ge 1$, and two community assignments $\theta,\eta\in\Theta_n$,
we are interested in a test that distinguishes one from the other and the
corresponding testing power.

\subsection{Existence of tests for community assignments}

We base the test on the likelihood ratio
$dP_{\eta}/dP_{\theta}$.
Fix $n\ge 1$, let $X^n$ denote the random graph associated with
$\theta\in\Theta_n$ and let \(m_\theta\) be the number of 1-labels of
\(\theta\), so \(\theta\in \Theta_{n,m_\theta}\). Let $\eta$ denote another
element of $\Theta_n$ and suppose \(\eta\in \Theta_{n ,m_\eta}\), for
some \(m_\eta \in\set{0,\ldots,\floor{n/2}}\) (which might or might not
be equal to \(m_\theta\)). 
Compare $p_{\theta}(X^n)$ with $p_{\eta}(X^n)$ in the likelihood
ratio. Based on the probability density for $P_{\theta}$ and the
definitions of the edge sets $D_1$ and $D_2$ of (\ref{eq:thesetsD}),
we define,
\[
  (S_n,T_n):=\Bigl(\sum\{X_{ij}:(i,j)\in D_{1}(\theta,\eta)\},
    \sum \{X_{ij}:(i,j)\in D_{2}(\theta,\eta)\}\Bigr),
\]
and note that,
\begin{equation}
  \label{eq:SnTn}
  (S_n,T_n)\sim\begin{cases}
  \text{Bin}(|D_{1}(\theta,\eta)|,p_n)\times\text{Bin}(|D_{2}(\theta,\eta)|,q_n),
    \quad\text{if $X^n\sim P_{\theta}$},\\
  \text{Bin}(|D_{1}(\theta,\eta)|,q_n)\times\text{Bin}(|D_{2}(\theta,\eta)|,p_n),
    \quad\text{if $X^n\sim P_{\eta}$}.
  \end{cases}
\end{equation}
Since $S_n$ and $T_n$ are independent, the likelihood ratio can be written in terms of
the moment generating functions for two binomial random variables: 
\begin{equation}
  \label{eq:pbmlikratio}
  \frac{p_{\eta}}{p_{\theta}}(X^n)
  = \biggl(\frac{1-p_n}{p_n}\,\frac{q_n}{1-q_n}\biggr)^{S_n-T_n}
    \biggl(\frac{1-q_n}{1-p_n}\biggr)^{|D_{1,n}|-|D_{2,n}|}.
\end{equation} % gecheckt en het klopt!
This gives rise to the following lemma:
\begin{lemma}
\label{lem:testingpower}
Let $n\geq1$, $\theta,\eta\in\Theta_n$ be given. Then there
exists a test function $\phi:\scrX_n\to[0,1]$ such that,
\[
  \begin{split}
  \pi_n(\theta)P_{\theta}\phi(X^n) + \pi_n(\eta) & P_{\eta}(1-\phi(X^n))\\
    &\leq
    \pi_n(\theta)^{1/2}\pi_n(\eta)^{1/2}\rho(p_n,q_n)^{|D_{1,n}|+|D_{2,n}|}.
  \end{split}
\]
\end{lemma}
\begin{proof}
The likelihood ratio test $\phi(X^n)$ has testing power bounded by
the Hellinger affinity (see \cite{LeCam86} and \cite[lemma 2.7]{Kleijn21}),
\[
  \pi_n(\theta)P_{\theta}\phi(X^n) + \pi_n(\eta) P_{\eta}(1-\phi(X^n))
\leq  \pi_n(\theta)^{1/2}\pi_n(\eta)^{1/2}
    P_{\theta}\Bigl(\frac{p_{\eta}}{p_{\theta}}(X^n)\Bigr)^{1/2}.
\]
The Hellinger affinity is bounded as follows,
\[
  \begin{split}
  P_{\theta}\Bigl(\frac{p_{\eta}}{p_{\theta}}(X^n)\Bigr)^{1/2}
    &= P_{\theta}
    \biggl(\frac{p_n}{1-p_n}\,\frac{1-q_n}{q_n}\biggr)^{\ft12(T_n-S_n)}
    \biggl(\frac{1-q_n}{1-p_n}\biggr)^{\ft12(|D_{1,n}|-|D_{2,n}|)}\\
    &= Pe^{\ft12\lambda_nS_n}\,Pe^{-\ft12\lambda_nT_n}
    \biggl(\frac{1-q_n}{1-p_n}\biggr)^{\ft12(|D_{1,n}|-|D_{2,n}|)},
  \end{split}
\]
where $\lambda_n:=\log(1-p_n)-\log(p_n)+\log(q_n)-\log(1-q_n)$. Using
the moment-generating function of the
binomial distribution, we conclude that,
\[
  \begin{split}
  &P_{\theta}\biggl(\frac{p_{\eta}}
    {p_{\theta}}(X^n)\biggr)^{1/2}
  =\Bigl(1-p_n
      +p_n\Bigl(\frac{1-p_n}{p_n}\,\frac{q_n}{1-q_n}\Bigr)^{1/2}
      \Bigr)^{|D_{1,n}|}\\
  &   \qquad\times\Bigl(1-q_n
      +q_n\Bigl(\frac{p_n}{1-p_n}\,\frac{1-q_n}{q_n}\Bigr)^{1/2}
      \Bigr)^{|D_{2,n}|}
      \biggl(\frac{1-q_n}{1-p_n}\biggr)^{\ft12(|D_{1,n}|-|D_{2,n}|)}\\
  &= \rho(p_n,q_n)^{|D_{1,n}|+|D_{2,n}|},
  \end{split}
\]
which proves the assertion. 
\end{proof}

\subsection{Lower bounds for the sizes of edge sets}
\label{sub:lowerboundsDs}

Testing power for one community assignment versus the other grows when the
edge sets $D_1$ and $D_2$ have many elements. It is therefore of interest
to find (sharp) lower bounds. To that end, note that $\set{1,\ldots,n}$ is
the disjoint union $V_{00}\cup V_{01}\cup V_{10}\cup V_{11}$, where
$V_{ab} = \set{i: \theta_i=a,\eta_i=b}$. In the edge sets we only
count pairs $(i,j)$ with $i<j$, so,
\[
  \begin{split}
  |D_{1}(\theta,\eta)| &
  %= \frac12(2|V_{00}|\cdot |V_{01}|+ 2|V_{11}|\cdot |V_{10}|)
    = |V_{00}|\cdot |V_{01}|+ |V_{11}|\cdot |V_{10}|,\\
  |D_{2}(\theta,\eta)| &= |V_{00}|\cdot |V_{10}|+ |V_{01}|\cdot |V_{11}|.
  \end{split}
\]
So that,
\[
  |D_{1}(\theta,\eta)|+|D_{2}(\theta,\eta)|
  = \bigl(|V_{00}|+|V_{11}|\bigr)\bigl(|V_{01}|+|V_{10}|\bigr).
\]
With $k=|V_{10}|+|V_{01}|$ (and using that $
|V_{00}|+|V_{01}|+|V_{10}|+|V_{11}|=n$), we find that $
|V_{00}|+|V_{11}|=n -k$, and we arrive at,
\begin{equation}
\label{eq:sizeofD1andD2inVnmk}
  |D_{1}(\theta,\eta)|+ |D_{2}(\theta,\eta)|= k(n-k). 
\end{equation}
Note that, 
\[
  \min_{\theta\in\Theta_{n,m_1},\eta\in\Theta_{n,m_2}}k(\theta,\eta)=|m_1-m_2|. 
\]
As $m_1,m_2\in\{0,\ldots,\floor{n/2}\}$, $|m_1-m_2|\in\{0,\ldots,\floor{n/2}\}$.
and since $k(n-k)$ is increasing in $k$ on $\set{0,\ldots,\floor{n/2}}$, we have,
\begin{equation}
  \label{eq:minovertwosetsofd}
  \min_{\theta\in\Theta_{n,m_1},\eta\in\Theta_{n,m_2}}
    (|D_{1}(\theta,\eta)|+ |D_{2}(\theta,\eta)|)
    = |m_1-m_2|(n-|m_1-m_2|). 
\end{equation}

\section{Auxiliary results}
\label{sec:aux}

\begin{lemma}
\label{lem:oneplusxdivrtothepowerrissmallerthanetothepowerx} 
For all positive integers \(r\) and real \(x>-r\), \((1+x/r)^r \le e^x\).
\end{lemma}
\begin{proof}
Let for \(x>-r\), \(f(x)= r\log(1+x/r)\) and \(g(x)=x\). Then \(f'(x)=(1+x/r)^{-1}\)
and \(g'(x)=1\). Then \(f'(x)\leq g'(x)\), when \(x\geq 0\), \(f'(x)>g'(x) \)
when \(-n<x<0\) and \(f(0)=g(0)\). It follows that \(f(x)\leq g(x)\) for all \(x>-r\).
As \(y\to e^y\) is increasing for all real \(y\), we find \(x>-n\),
\((1+x/r)^r = e^{f(x)}\le e^{g(x)}=e^x\).
\end{proof}
\begin{lemma}
\label{lem:boundforbinomialsum}
For $x\in[0,1]$, 
\[
  \sum_{k=1}^{\floor{n/2}} \binom{n}{k} x^{k(n-k)}
  \leq 2((1+x^{n/2})^n -1) \leq 2nx^{n/2}  e ^ {nx^{n/2}}.  
\]
\end{lemma}
\begin{proof}
Define $a_k=\binom{n}{k}x^{k(n-k)}$ and note that $a_k=a_{n-k}$. Since $x\in[0,1]$
and $n-k\geq n/2$ for all $k\in\set{1,\ldots,\floor{n/2}}$, Newton's binomium
gives rise to,  
\[
  \sum_{k=1}^{\floor{n/2}} \binom{n}{k} x^{k(n-k)} 
  %\leq 2\sum_{k=1}^{\floor{n/2}} \binom{n}{k} x^{kn/2}
  \leq \sum_{k=1}^{\floor{n/2}} \binom{n}{k} x^{kn/2}
  \leq \bigl((1+x^{n/2})^n -1\bigr) \leq nx^{n/2}  e ^ {nx^{n/2}}.
\]
where the last inequality is based on
lemma~\ref{lem:oneplusxdivrtothepowerrissmallerthanetothepowerx}.
\end{proof}

\end{document}